\documentclass[12pt,twoside]{amsart}
\usepackage{amsxtra}
\usepackage{color}
\usepackage{amsopn}
\usepackage{amsmath,amsthm,amssymb}
\title{Einstein locally conformal   calibrated $G_2$-structures}
\author{Anna Fino and Alberto Raffero}
\thanks{Research partially supported by  the project PRIN {\em Variet\`a reali e complesse: geometria, topologia e analisi armonica},  the project FIRB {\em Differential Geometry and Geometric functions theory} and by  GNSAGA (Indam) of Italy.}
\address{Dipartimento di Matematica, Universit\`a di Torino, via Carlo Alberto 10, 10123 Torino, Italy}
\email{annamaria.fino@unito.it}
\email{alberto.raffero@unito.it}

\theoremstyle{remark}
\newtheorem{remark}{Remark}[section]

\theoremstyle{definition}

\newtheorem{ex}[remark]{Example}

\theoremstyle{plain}
\newtheorem{teo}[remark]{Theorem}
\newtheorem{prop}[remark]{Proposition}
\newtheorem{corol}[remark]{Corollary}

\newcommand{\beq}{\begin{equation}}
\newcommand{\eeq}{\end{equation}}
\newcommand{\bqn}{\begin{eqnarray}}
\newcommand{\eqn}{\end{eqnarray}}
\newcommand{\bqne}{\begin{eqnarray*}}
\newcommand{\eqne}{\end{eqnarray*}}

\newcommand{\R}{{\mathbb R}}

\newcommand{\tz}{\tau_0}
\newcommand{\tu}{\tau_1}
\newcommand{\td}{\tau_2}
\newcommand{\ttr}{\tau_3}
\newcommand{\W}{\wedge}

\newcommand{\f}{\varphi}
\newcommand{\de}{\delta}
\newcommand{\RR}{\mathbb{R}}

\newcommand{\SU}{{\rm SU}}

\begin{document}
\maketitle

\begin{abstract} 
We study locally conformal calibrated $G_2$-structures whose underlying Riemannian metric is Einstein, showing that in the compact case the scalar curvature cannot be positive.
As a consequence, a compact homogeneous  7-manifold  cannot admit  an invariant  Einstein locally  conformal   calibrated $G_2$-structure unless the underlying metric is flat. 
In contrast to the compact case, we  provide a non-compact example of  homogeneous manifold endowed with a  locally conformal calibrated $G_2$-structure   
whose associated  Riemannian metric is Einstein and non Ricci-flat.  
The  homogeneous Einstein metric  is a rank-one  extension of a Ricci soliton on the 3-dimensional complex Heisenberg group endowed with a left-invariant coupled  
$\SU(3)$-structure $(\omega, \Psi)$, i.e.,  such that $d \omega = c {\rm Re}(\Psi)$, with $c \in \R - \{ 0 \}$. 
Nilpotent Lie algebras admitting  a  coupled  $\SU(3)$-structure are also classified.
\end{abstract}

\section{Introduction}

We recall that a seven-dimensional smooth manifold $M$ admits a $G_2$-structure if the  structure group of the frame bundle reduces 
to the exceptional Lie group $G_2$. 
The existence of a $G_2$-structure is equivalent to the existence of a non-degenerate 3-form $\varphi$ defined on the whole manifold (see for example \cite{Jo3}) and 
using this 3-form it is possible to define a Riemannian metric $g_\f$ on $M$.

If $\varphi$ is parallel with respect to the Levi-Civita connection, i.e., $\nabla^{LC} \varphi =0$, then the holonomy group is contained in $G_2$, the $G_2$-structure is called {\it parallel} 
and the corresponding manifolds are called {\em $G_2$-manifolds}. In this case, the induced metric $g_\f$ is Ricci-flat.
The first examples of complete metrics with holonomy $G_2$ were constructed by Bryant and Salamon \cite{BS}. 
Compact examples of manifolds with holonomy $G_2$ were obtained first by Joyce \cite{Jo1, Jo2, Jo3} and  then by Kovalev \cite{Kov} and by Corti, Haskins, 
Nordstr\"om, Pacini \cite{CHNP}. 
Incomplete Ricci-flat metrics of holonomy $G_2$  with a 2-step nilpotent isometry group $N$  acting on orbits of codimension 1 were obtained in \cite{CF, GLPS}. 
It turns out that these metrics are locally isometric (modulo a conformal change) to homogeneous metrics on solvable Lie groups, which are obtained as rank one extensions of  
a six-dimensional nilpotent Lie group  endowed with an invariant $\SU(3)$-structure of a special kind, known in the literature as 
{\it half-flat} \cite{CS}.

Examples of compact and non-compact  manifolds endowed with non-parallel  $G_2$-structures were given for instance in
\cite{Ca, CMS, CF, Fe1, Fe}. In particular, in \cite{CF} conformally parallel $G_2$-structures on solvmanifolds, i.e., on simply connected solvable Lie groups, were studied. 
More in general,  in \cite{IPP} it was shown that a seven-dimensional compact Riemannian manifold $M$ admits a locally conformal
parallel $G_2$-structure if and only if it has as   covering  a Riemannian cone over a compact nearly K\"ahler 6-manifold  such that the covering transformations are
homotheties preserving the corresponding parallel $G_2$-structure. 

By \cite{Br, CI, FI}, it is evident that the Riemannian scalar curvature of a $G_2$-structure may be expressed in terms of the $3$-form 
$\varphi$ and its derivatives.  More precisely, in \cite{Br} an expression of the Ricci curvature and the scalar curvature in terms of the four
intrinsic torsion forms $\tau_i, i =0, \ldots, 3$, and their exterior derivatives was given. 
Moreover,  using this it is  possible to show that  the scalar curvature has a definite sign for certain classes of
$G_2$-structures. 

If $d\f = 0$, the $G_2$-structure is called {\it calibrated} or {\it closed}. 
The geometry of this family of $G_2$-structures was studied in \cite{CI}. 
Furthermore, Bryant proved in \cite{Br} that if the scalar curvature of a closed $G_2$-structure is non-negative then the $G_2$-structure is parallel.

We say that a $G_2$-structure $\varphi$ is {\it Einstein} if the underlying Riemannian metric $g_{\varphi}$ is Einstein. 
In \cite{Br,CI} it was proved, as an analogous of Goldberg conjecture for almost-K\"ahler manifolds,  that  on a compact manifold an Einstein (or, more in general,  
with divergence-free Weyl tensor \cite{CI}) calibrated $G_2$-structure has holonomy contained in $G_2$. 
In the non-compact case, Cleyton and Ivanov  showed  that the same result is true with the additional assumption that the $G_2$-structure is $*$-Einstein, 
but it still an open problem to see if there exist (even incomplete) Einstein metrics underlying  calibrated $G_2$-structures.   
Recently, some negative results were proved in the case of  non-compact homogeneous spaces in \cite{FFM}. 
In particular, the authors showed that a seven-dimensional solvmanifold cannot admit any left-invariant calibrated $G_2$-structure inducing an Einstein metric 
$g_\f$ unless $g_\f$ is flat.

In the present paper, we are mainly interested in the geometry of {\it locally  conformal calibrated} $G_2$-structures, i.e., $G_2$-structures whose associated metric is 
conformally equivalent (at least locally) to the metric induced by a calibrated $G_2$-structure.

In Section \ref{Einsteinlocconfcal}, we  prove that a compact manifold endowed with an Einstein locally conformal calibrated $G_2$-structure has non-positive scalar curvature 
(and then has either zero or negative curvature if it is also connected) and we show that  a compact homogeneous 7-manifold cannot admit an invariant Einstein locally conformal 
calibrated $G_2$-structure unless the underlying metric is flat.   

In the last section, we give a non-compact example of a homogeneous manifold endowed with an Einstein locally conformal calibrated $G_2$-structure. 
The homogeneous manifold is a solvmanifold, thus this example and the aforementioned result of \cite{FFM} highlight a different behaviour of 
calibrated and locally conformal calibrated $G_2$-structures.   
Moreover, the homogeneous Einstein metric is a rank-one extension of a Ricci soliton on the complex Heisenberg group 
induced by a coupled $\SU(3)$-structure $(\omega, \Psi)$ such that $d \omega = -  {\rm Re} (\Psi)$. 
Recall that a half-flat $\SU(3)$-structure is said to be {\it coupled} if $d\omega$ is proportional to 
${\rm Re} (\Psi)$ at each point (see \cite{Sa}). 
Finally, we classify nilpotent Lie groups  admitting a  left-invariant coupled $\SU(3)$-structure, showing that the complex Heisenberg group  is, up to isomorphisms, 
the only nilpotent Lie group admitting  a  coupled  $\SU(3)$-structure $(\omega, \Psi)$  whose associated metric is a Ricci soliton.

\medskip
\noindent {\em{Acknowledgements}}.  The authors would like to thank Simon Salamon,  Uwe Semmelmann and Luigi Vezzoni   for  useful  comments on the paper. Aspects of
the work were described at the Workshop  \lq \lq $G_2$ days\rq \rq (King's and University College, London, 2012). 
 
\smallskip
\section{Preliminaries on $G_2$ and $\SU(3)$-structures} 

Let $\left(e_1, \ldots, e_7\right)$ be the standard basis of $\R^7$ and $\left(e^1, \ldots, e^7\right)$ be the corresponding dual basis. We set
$$
\varphi = e^{123}  + e^{145}  + e^{167}  + e^{246}  - e^{257}  - e^{347} - e^{356},
$$
where for simplicity $e^{ijk}$ stands for the wedge product $e^i \wedge e^j \wedge e^k$ in $\Lambda^3 ({(\R^7)}^*)$.
The  subgroup of ${\rm GL}(7, \R)$ fixing $\varphi$ is  $G_2$.
The basis  $(e^1,\ldots, e^7)$  is an oriented orthonormal basis for the underlying metric and the orientation is determined by the inclusion $G_2\subset {\rm SO}(7)$. 
The group $G_2$ also fixes the 4-form
$$
* \varphi = e^{4567} + e^{2367} + e^{2345} + e^{1357} - e^{1346} - e^{1256} - e^{1247},
$$
where $*$ denotes the Hodge star operator determined by the  associated metric  and orientation.

We recall that a $G_2$-structure on a $7$-manifold  $M$ is characterized by a positive 3-form $\f$. Indeed, it
turns out that there is a $1-1$ correspondence between $G_2$-structures on a 7-manifold
and  3-forms  for which the bilinear form  $B_{\f}$   defined by
$$
B_{\f} (X, Y) = \frac{1}{6}  \, i_{X} \f \wedge i_Y\f \wedge  \f 
$$
is positive definite, where  $i_{X}$  denotes  the contraction by $X$. A 3-form  $\f$  for which $B_{\f}$ is positive definite defines a unique Riemannian metric $g_{\f}$
and volume form $dV_\f$ such that for any couple of vectors $X$ and $Y$ on $M$ the following relation holds
$$
g_{\f} (X, Y) dV_\f  = \frac  {1}{6} \,  i_{X} \f \wedge i_Y\f \wedge  \f.
$$
As in \cite{CI}, we let 
$$
\varphi = \frac{1} {6} \varphi_{ijk} e^{ijk}
$$
and define the $*$-Ricci tensor of the $G_2$-structure as
$$
\rho^*_{sm} := R_{ijkl} \varphi_{ijs} \varphi_{klm}.
$$
A $G_2$-structure is said to be $*$-Einstein if the traceless part of the $*$-Ricci tensor vanishes, i.e., if 
$\rho^* = \frac{s^*}{7} g$, where $s^*$ is the trace of $\rho^*$.

On a $7$-manifold endowed with a $G_2$-structure, the action of $G_2$ on the tangent spaces induces an action of $G_2$ on the exterior algebra ${\Lambda}^p(M)$, 
for any $p \geq 2$.  In \cite{Br2}, it was shown that there are irreducible $G_2$-module decompositions
$$
\begin{array} {l}
\Lambda^2 ({(\R^7)}^*) = \Lambda^2_{7} ({(\R^7)}^*) \oplus
 \Lambda^2_{14} ({(\R^7)}^*),\\[3 pt]
\Lambda^3 ({(\R^7)}^*) = \Lambda^3_{1} ({(\R^7)}^*)\oplus \Lambda^3_7 ({(\R^7)}^*)
\oplus \Lambda^3_{27} ({(\R^7)}^*),
\end{array}
$$
where $\Lambda^p_k ({(\R^7)}^*)$ denotes an irreducible $G_2$-module of dimension $k$. 
Using the previous decomposition of $p$-forms, in
\cite{Br} a simple expression of $d\varphi$  and $d*\varphi$ was obtained, 
where $*$ denotes the Hodge operator defined  by  the metric $g_\f$ and the volume form $dV_\f$.
More precisely,  for any $G_2$-structure $\varphi$ there exist unique differential forms $\tau_0 \in \Lambda^0 (M),$ 
$\tau_1 \in \Lambda^1 (M),$   $\tau_2 \in \Lambda^2_{14} (M),$ 
$\tau_3 \in \Lambda^3_{27} (M),$ 
such that
$$
\begin{array}{rcl}
d \varphi& = & \tau_0 *  \varphi + 3 \tau_1 \wedge \varphi + * \tau_3,\\[3 pt]
d * \varphi & = & 4 \tau_1 \wedge * \varphi + \tau_2 \wedge \varphi,
\end{array}
$$
where $\Lambda^p_k (M)$ denotes the space of sections of the bundle $\Lambda^p_k (T^* M)$.

In the case of a closed $G_2$ structure we have
$$
\begin{array} {rcl}
 d \varphi & = & 0,\\
d*\varphi & = &  \tau_2 \wedge \varphi.
\end{array}
$$
By the results of \cite{Br}, the scalar curvature is given by 
$$
{\rm Scal} (g_{\varphi}) =  - \frac {1} {2} | \tau_2|^2
$$
and from this it is clear that it cannot be positive.

For a locally conformal calibrated $G_2$-structure $\f$ one has $\tz \equiv 0$ and $\ttr\equiv0$, so
$$ 
\begin{array} {rcl}
 d \varphi & = & 3 \tau_1  \wedge \varphi,\\
d*\varphi & = & 4 \tau_1 \wedge * \varphi + \tau_2 \wedge \varphi,
\end{array}
$$
and taking the exterior derivative of the former it is easy to show that $\tu$ is a closed 1-form. 
Moreover, in this case the scalar curvature has not a definite sign as one can check from its expression
$$
{\rm Scal} (g_{\varphi}) = 12 \de \tau_1 + 30 | \tau_1 |^2 - \frac {1} {2} | \tau_2 |^2,
$$
where $\de$ denotes the adjoint of the exterior derivative $d$ with respect to the metric $g_\f$.

If the only nonzero intrinsic torsion form is $\tu$, we have the so called {\it locally conformal parallel} $G_2$-structures. 
They are named in this way since a conformal change of the metric $g_\f$ associated to a $G_2$-structure of this kind gives (at least locally) the metric 
induced by a parallel $G_2$-structure. In this case
$$ 
\begin{array} {rcl}
 d \varphi & = & 3 \tau_1  \wedge \varphi,\\
d*\varphi & = & 4 \tau_1 \wedge * \varphi.
\end{array}
$$
We will give an example of such a structure at the end of  Section \ref{sectex}.

We recall that a six-dimensional smooth manifold admits an $\SU(3)$-structure if the structure group of the frame bundle can be reduced to $\SU(3)$.
It is possible to show that the existence of an $\SU(3)$-structure is equivalent to the existence of an almost  Hermitian structure  $(h, J, \omega)$ and a unit $(3,0)$-form  $\Psi$.

Since $\SU(3)$ is the stabilizer of the transitive action of $G_2$ on the 6-sphere $S^6,$ it follows that a $G_2$-structure on a 7-manifold
induces an $\SU(3)$-structure on any oriented hypersurface. 
If the $G_2$-structure is parallel, then the $\SU(3)$-structure is half-flat \cite{CS}. In terms of the  forms $(\omega,  \Psi)$ this means 
$d (\omega \wedge \omega) =0, d  ({\rm Re}(\Psi))=0$. 
 
In our computations we will use another characterization of $\SU(3)$-structures which follows from the results of \cite{Hit, Rei}. We describe it here.
Consider a six-dimensional oriented real vector space $V$, a $k$-form on $V$ is said to be {\it stable} if its GL(V)-orbit is open. 
Let $A:\Lambda^5(V^*) \rightarrow V\otimes\Lambda^6(V^*)$ denote the canonical isomorphism  given by $A(\gamma) = w \otimes \Omega$, where $i_w\Omega = \gamma$, and
define for a fixed 3-form $\sigma \in\Lambda^3(V^*)$
$$
K_\sigma : V \rightarrow V\otimes\Lambda^6(V^*),\ \  K_\sigma(w) = A((i_w \sigma)\W\sigma)
$$ 
and 
$$
\lambda : \Lambda^3(V^*) \rightarrow (\Lambda^6(V^*))^{\otimes2},\ \  \lambda(\sigma) = \frac16{\rm tr}K^2_\sigma.
$$
A 3-form $\sigma$ is stable if and only if $\lambda(\sigma)\neq0$ and whenever this happens it is possible to define a volume form by $\sqrt{|\lambda(\sigma)|} \in \Lambda^6(V^*)$,  
where the positively oriented root is chosen, and an endomorphism 
$$
J_\sigma = \frac{1}{\sqrt{|\lambda(\sigma)|}}K_\sigma,
$$ 
which is a complex structure when $\lambda(\sigma)<0$.  

A pair of stable forms $(\omega,\sigma)\in\Lambda^2(V^*)\times\Lambda^3(V^*)$ is called {\it compatible} if $\omega\W\sigma=0$ and {\it normalized} if 
$J^*_\sigma\sigma\W\sigma = \frac23\omega^3$ (the latter identity is non-zero since a 2-form $\omega$ is stable if and only if $\omega^3\neq0$). 
Such a pair defines a (pseudo) Euclidean metric $h(\cdot,\cdot) = \omega(J_\sigma\cdot,\cdot)$. 
As a consequence, on a six-dimensional smooth manifold $N$ there is a one to one correspondence between 
$\SU(3)$-structures and pairs $(\omega,\sigma)\in\Lambda^2(N)\times\Lambda^3(N)$ such that for each point $p\in N$ the pair of forms defined on $T_pN$  
$(\omega_p,\sigma_p)$ is stable, compatible, normalized, has $\lambda(\sigma_p)<0$ and induces a Riemannian metric $h_p(\cdot,\cdot) = \omega_p(J_{\sigma_p}\cdot,\cdot)$. 
In this case we have $\Psi = \sigma + {\rm i}J^*_\sigma\sigma$ and, then, $\sigma = {\rm Re}(\Psi)$. 
We refer to $h$ as the {\it associated Riemannian metric} to the $\SU(3)$-structure $(\omega,\sigma)$.

An $\SU(3)$-structure $(\omega,\sigma)$ on a $6$-manifold $N$ is called {\it coupled} if $d \omega = c \sigma$, with $c$ a non-zero real number.  
Note that in particular a coupled $\SU(3)$-structure is half-flat since $d (\omega^2) =0$ and $d \sigma =0$ and its intrinsic torsion belongs to the space 
${\mathcal W_1}^- \oplus {\mathcal W_2}^- $, where ${\mathcal W_1}^- \cong \R$ and ${\mathcal W_2}^- \cong  \frak{su} (3)$ (see \cite{CS}). 

It is interesting to notice that the product manifold $N \times \RR$, where $N$ is a $6$-manifold endowed with a coupled $\SU(3)$-structure $(\omega, \sigma)$, 
has a natural locally conformal calibrated $G_2$-structure defined by 
$$\varphi  = \omega \wedge dt + \sigma.$$ 
Indeed,
$$
d \varphi = c \sigma  \wedge dt =  c \varphi \W dt,
$$
since in local coordinates the components of $\sigma$ are functions defined on $N$ and thus they do not depend on $t$. 
Then, $\tz \equiv 0, \ttr \equiv 0$ and $\tu = \left(-\frac13c\right)dt$. 

\smallskip
\section{Einstein locally conformal calibrated $G_2$-structures on compact manifolds} \label{Einsteinlocconfcal}
We will show now that a seven-dimensional, compact, smooth manifold $M$ endowed with an Einstein  locally conformal calibrated  
$G_2$-structure $\varphi$ has ${\rm Scal} (g_{\varphi}) \leq 0$. 
It is worth observing here that, up to now, there are no known examples of smooth manifolds endowed with a locally conformal calibrated 
$G_2$-structure whose associated metric is Ricci-flat (and then has zero scalar curvature).

First of all recall that given a Riemannian manifold $(M,g)$ of dimension $n \geq 3$ it is possible to define the so called {\it conformal Yamabe constant} $Q(M,g)$ in the following way: 
set $a_n := \frac{4(n-1)}{n-2}, p_n := \frac{2n}{n-2}$ and let $C^\infty_c(M)$ denote the set of compactly supported smooth real valued functions on $M$. Then 
$$
Q(M,g) :=  \underset{u\in C^\infty_c(M), u \not\equiv 0}{\rm inf} \left\{\frac{\int_M(a_n|du|_g^2 + u^2 {\rm Scal}(g))dV_g}{(\int_M|u|^{p_n}dV_g)^{\frac2{p_n}}} \right\} .
$$ 

The sign of $Q(M,g)$ is a conformal invariant, in particular the following characterization holds:
\begin{prop}
If $(M,g)$ is a compact Riemannian manifold of dimension $n \geq 3$, then $Q(M,g)$ is negative/zero/positive if and only if $g$ is conformal to a Riemannian metric of  
negative/zero/positive scalar curvature.
\end{prop}

Using the conformal Yamabe constant it is possible to prove the following

\begin{teo}
Let $M$ be a seven-dimensional, compact, smooth manifold endowed with an Einstein locally conformal calibrated   
$G_2$-structure $\varphi$. Then ${\rm Scal} (g_{\varphi})\leq0$. Moreover, if $M$ is connected, ${\rm Scal}(g_\f)$ is either zero or negative.
\end{teo}

\proof
Suppose that ${\rm Scal}(g_\f) >0$, then the 1-form $\tu$ is exact. Indeed, since $d\tu=0$, we can consider the de Rham class $[\tu] \in H^1_{{\rm dR}} (M)$ and take 
the harmonic 1-form $\xi$ representing $[\tu]$, that is, $\tu = \xi + df$, where $\Delta\xi = 0$ and $f\in C^\infty(M)$.
$\xi$ has to vanish everywhere on $M$ since it is compact, oriented and has positive Ricci curvature. Then $\tu =df$.
Let us consider $\tilde\f := e^{-3f}\f$, it is clear that $\tilde\f$ is a $G_2$-structure defined on $M$. Moreover
$$
\begin{array}{lcl}
d\tilde\f & = & d(e^{-3f}\f) \\
              & = & -3e^{-3f}df\W\f + e^{-3f}d\f \\
              & = & -3e^{-3f}\tu\W\f + e^{-3f}(3\tu\W\f) \\
              & = & 0,
\end{array}
$$
so $\tilde\f$ is a closed $G_2$-structure and ${\rm Scal}(g_{\tilde\f}) \leq 0$ by \cite{Br}. 
We have $g_{\tilde\f} = e^{-2f}g_\f$, that is, $g_{\tilde\f}$ is conformal to the Riemannian metric $g_\f$ of positive scalar curvature, then the conformal Yamabe constant 
$Q(M, g_{\tilde\f})$ is positive by the previous characterization.

Since $M$ is compact, it has finite volume and is complete as a consequence of the well known Hopf-Rinow Theorem. Then, by \cite[Corollary 2.2]{Leu} we have that 
$Q(M, g_{\tilde\f}) \leq 0$, which is in contrast with the previous result.
\endproof

As a consequence of the previous proposition we have the
\begin{corol}
A seven-dimensional, compact, homogeneous, smooth manifold $M$ cannot admit an invariant locally conformal calibrated Einstein $G_2$-structure 
$\varphi$, unless the underlying metric $g_{\varphi}$  is flat.
\end{corol}
\proof
Recall that a homogeneous Einstein manifold with negative scalar curvature is not compact \cite{Be}. Thus, every seven-dimensional, compact, 
homogeneous, smooth manifold $M$  with an invariant  $G_2$-structure $\f$ whose associated metric is Einstein has ${\rm Scal}(g_{\varphi}) \geq 0$. 
Combining this result with the previous proposition we have ${\rm Scal} (g_{\varphi}) = 0$ and, in particular, $g_\f$ is Ricci-flat.  
The statement then follows recalling that in the homogeneous case Ricci flatness implies flatness \cite{AK}.
\endproof

\smallskip
\section {Noncompact homogeneous examples and coupled  $\SU(3)$-structures}\label{sectex}
In this section, after recalling some facts about noncompact homogeneous Einstein manifolds, we first study the classification of coupled $\SU(3)$-structures on nilmanifolds 
and then we construct an example of a locally conformal calibrated  $G_2$-structure $\varphi$  inducing  an Einstein (non Ricci-flat)  metric on a noncompact homogeneous manifold.

All the known examples of noncompact homogeneous Einstein manifolds  are solvmanifolds, i.e., simply connected solvable Lie groups $S$ endowed with a left-invariant metric 
(see for instance the recent survey \cite{Lauret1}). D. Alekseevskii conjectured that  these  might exhaust the class of non-compact homogeneous Einstein manifolds  
(see \cite[7.57]{Be}).

Lauret in \cite{Lauret2} showed that  every  Einstein solvmanifold is \emph{standard}, i.e., it is 
a  solvable Lie group $S$ endowed with  a left-invariant metric such that  the orthogonal complement 
${\frak a}=[{\frak s},{\frak s}]^{{\perp}}$, where ${\frak s}$ is the Lie algebra of $S$, is abelian.
We recall that given a metric nilpotent Lie algebra ${\frak n}$ with an inner product $\langle \cdot , \cdot \rangle_\frak{n}$, a  metric solvable Lie
algebra $({\frak s} = {\frak n} \oplus {\frak a}, \langle \cdot , \cdot \rangle_\frak{s})$ is called a {\em metric solvable extension} of  $({\frak n}, \langle \cdot , \cdot \rangle_\frak{n})$
if $[\frak{s},\frak{s}] = \frak{n}$ and the restrictions to ${\frak n}$ of the Lie bracket of ${\frak s}$ and of the inner product $\langle \cdot , \cdot \rangle_\frak{s}$ coincide with
the Lie bracket of ${\frak n}$ and with $\langle \cdot , \cdot \rangle_\frak{n}$, respectively.
The dimension of $\mathfrak a$ is called the {\it algebraic rank} of $\mathfrak s$.

In \cite[4.18]{He}, it was  proved that the study of standard Einstein  metric solvable Lie algebras  reduces to the  rank-one
metric solvable extension of a nilpotent Lie algebra (i.e., those for which $\dim ({\frak a}) = 1$). 
Indeed, by \cite{He} the metric Lie algebra of any $(n + 1)$-dimensional rank-one solvmanifold can be modelled  on 
$({\frak s} = \frak{n}\oplus \RR H, \langle \cdot , \cdot \rangle_\frak{s})$ 
for some nilpotent Lie algebra  $\frak n$, with the inner product $\langle \cdot , \cdot \rangle_\frak{s}$ such that $ \langle H, {\frak n} \rangle_\frak{s} = 0$, 
$\langle H,H\rangle_\frak{s} = 1$ and the Lie bracket on $\frak s$ given by 
$$
[H, X]_\frak{s} = D X, \quad [X, Y]_\frak{s}= [X, Y]_{\frak n},
$$
where $[\cdot , \cdot ]_{\frak n}$ denotes the Lie bracket on ${\frak n}$ and $D$ is some derivation of $\frak n$.  By \cite{Lauret3}, a left-invariant metric $h$ on a nilpotent Lie group $N$ 
is a Ricci soliton if and only if the Ricci operator satisfies ${\rm Ric}(h) = \mu I + D$, for some $\mu \in \R$ and some derivation $D$ of $\frak n$, 
when $h$ is identified with an inner product on $\frak{n}$ 
or, equivalently, if and only if $(N, h)$ admits a metric standard extension whose corresponding standard solvmanifold is Einstein.  The inner product $h$ is also called {\it nilsoliton}.

In \cite{Wi}, all the seven-dimensional rank-one Einstein solvmanifolds were determined, 
proving that each one of the 34 nilpotent Lie algebras $\frak n$ of dimension 6 admits a rank-one solvable extension which can be endowed with an Einstein  inner product.

Six-dimensional nilpotent Lie algebras admitting  a half-flat $\SU(3)$-structure were classified in \cite{Conti}.
For coupled  $\SU(3)$-structures  we can  show the following 

\begin{teo} Let $\frak n$ be a six-dimensional, non-abelian, nilpotent Lie algebra admitting a coupled $\SU(3)$-structure. Then $\frak n$ 
is isomorphic to one of the following
$$
\frak n_9 = (0,0,0,e^{12},e^{14}-e^{23},e^{15}+e^{34}),   \quad  \frak n_{28}=(0,0,0,0,e^{13}-e^{24},e^{14}+e^{23}),
$$
where for instance $ \frak n_9 = (0,0,0,e^{12},e^{14}-e^{23},e^{15}+e^{34})$ means that there exists a basis 
$(e^1, \ldots, e^6)$ of $\frak n_9^*$ such that 
$$
d e^j =0, j = 1,2,3,   \quad d e^4 = e^{12},  \quad d e^5 = e^{14}-e^{23},  \quad d e^6  = e^{15}+e^{34}.
$$
Moreover, the only nilpotent Lie algebra admitting  a  coupled  $\SU(3)$-structure inducing a nilsoliton is $\frak n_{28}$.
\end{teo}

\begin{proof}  
By the results in \cite{Conti}, the generic nilpotent Lie algebra $\mathfrak{n}$ admitting a half-flat $\SU(3)$-structure is isomorphic to one of the 24 Lie algebras described  
in Table \ref{tab:tabella}. Consider on $\mathfrak{n}$ a generic 2-form 
\begin{eqnarray}
\omega & = & b_1e^{12} + b_2e^{13}+b_3e^{14}+b_4e^{15} + b_5e^{16} + b_6e^{23} + b_7 e^{24} +b_8e^{25}\nonumber\\
               &  &  +b_9e^{26} + b_{10}e^{34} + b_{11}e^{35} + b_{12}e^{36} +b_{13}e^{45} + b_{14} e^{46} + b_{15}e^{56},\nonumber
\end{eqnarray}
where $b_i\in\RR, i=1,\ldots,15$,   and the  3-form 
$$\sigma = c(d\omega),  \quad c\in\RR-\{0\}.$$
The expression of  $\lambda(\sigma)$ for each nilpotent Lie algebra considered is given in Table \ref{tab:tabella}.

We observe that among the 24 nilpotent Lie algebras admitting a half-flat $\SU(3)$-structure we have:
\begin{itemize}
\item 1 case $(\mathfrak{n}_{28})$ for which $\lambda(\sigma)<0$ if $b_{15}\neq 0$,
\item 2 cases $(\mathfrak{n}_4$ and $\mathfrak{n}_9)$ for which the sign of $\lambda(\sigma)$ depends on $\omega$,
\item 21 cases for which $\lambda(\sigma)$ cannot be negative.
\end{itemize}

Therefore,  the 21 algebras having $\lambda(\sigma) \geq 0$ do not admit any  coupled $\SU(3)$-structure. 

Consider $\frak{n}_4$, it has structure equations
$$(0,0,e^{12},e^{13}, e^{14}+e^{23}, e^{24}+e^{15}).$$
First of all, observe that if $b_{15}=0$ then $\lambda(\sigma)=0$. So if we want to find an $\SU(3)$-structure we have to look for $\omega$ with $b_{15}\neq 0$. 
Moreover, $\sigma$ induces an almost complex structure if and only if $\lambda(\sigma)$ is negative, 
then we have to suppose in addition that $b_{15}(b_{12}+b_{13})>b^2_{14}$.
Since we want $\omega$ to be the 2-form associated to an $\SU(3)$-structure, it must be a form of type $(1,1)$ and this happens if and only if 
$\omega(\cdot,\cdot)=\omega(J\cdot,J\cdot)$, where $J = J_{\sigma}$.
Computing the previous identity with respect to the considered frame, we have that the following equations have to be satisfied by the 
components of $\omega$:
$$
\omega_{ab} =\sum_{k,m=1}^6 J^k_aJ^m_b\omega_{km},\quad  1\leq a < b \leq6
$$
(observe that $\omega_{12}=b_1$, $\omega_{13}=b_2$ and so on). 
Using these equations it is possible to write four of the $b_i$ in terms of the remaining and obtain a new expression for $\omega$.
We can now compute the matrix associated to $h(\cdot,\cdot) = \omega(J\cdot,\cdot)$ with respect to the basis $(e_1,\ldots,e_6)$ and observe that for the nonzero vector 
$v = e_4 -\frac{b_{14}}{b_{15}}e_5 +\frac{b_{13}}{b_{15}}e_6$ we have $h(v,v)=0$. 
Therefore, $h$ cannot be positive definite and, as a consequence, it is not possible to find a coupled $\SU(3)$-structure on $\frak{n}_4$.

For the Lie algebras  $\frak n_9$ and $\frak n_{28}$ we can give an explicit example of coupled $\SU(3)$-structure.
Consider on $\frak n_9$ the forms 
\begin{eqnarray}
\omega & = & -\frac32e^{12}-\frac14e^{14}-e^{15}-e^{24}+\frac12e^{26}-\frac12e^{35}-e^{36}+e^{56}, \nonumber \\
\sigma   & = & \frac{\sqrt{15}\sqrt[4]{2}}{4}e^{123}  +\frac{\sqrt {15}\sqrt [4]{2}}{8}e^{234}  -\frac{\sqrt {15}\sqrt [4]{2}}{8}e^{125} +\frac{\sqrt {15}\sqrt [4]{2}}{8}e^{134}\nonumber\\
               & &   +\frac{\sqrt {15}\sqrt [4]{2}}{4}e^{135}  -\frac{\sqrt {15}\sqrt [4]{2}}{4}e^{146}  +\frac{\sqrt {15}\sqrt [4]{2}}{4}e^{236}  +\frac{\sqrt {15}\sqrt [4]{2}}{4}e^{345}. \nonumber
\end{eqnarray}
We have  
$$\omega\W\sigma=0, \quad \omega^3\neq 0, \quad \lambda(\sigma) =-\frac{225}{64}, \quad d\omega = -\frac4{\sqrt{15}\sqrt[4]{2}}\sigma,$$
in particular $(\omega,\sigma)$ is a compatible pair of stable forms.
The associated almost complex structure  $J = J_{\sigma}$  has  the following  matrix expression  with respect to the basis $(e_1,\ldots,e_6)$:  
$$J=
 \left[ \begin {array}{cccccc} 0&0&-\sqrt {2}&0&0&0
\\ \noalign{\medskip}\sqrt {2}&0&0&-\sqrt {2}&0&0\\ \noalign{\medskip}
\frac{\sqrt {2}}{2}&0&0&0&0&0\\ \noalign{\medskip}0&\frac{\sqrt {2}}{2}&-\sqrt {
2}&0&0&0\\ \noalign{\medskip}\sqrt {2}&0&\frac{\sqrt {2}}{2}&\frac{\sqrt {2}}{2}
&0&\sqrt {2}\\ \noalign{\medskip}-\frac{\sqrt {2}}{4}&-\frac{\sqrt {2}}{4}&\frac{
3\sqrt {2}}{2}&0&-\frac{\sqrt {2}}{2}&0\end {array} \right] 
$$
and  it is easy to check that $J^*\sigma\W\sigma = \frac23\omega^3$, i.e., the pair $(\omega,\sigma)$ is normalized.

The inner product  $h(\cdot,\cdot) = \omega(J\cdot,\cdot)$   is given with respect to the basis $(e_1,\ldots,e_6)$ by 
$$h=
 \left[ \begin {array}{cccccc} \frac{5\sqrt {2}}{2}&\frac{\sqrt {2}}{8}&
\frac{\sqrt {2}}{4}&-\sqrt {2}&0&\sqrt {2}\\ \noalign{\medskip}\frac{\sqrt {2}}{8}&\frac{5\sqrt {2}}{8}&-\frac{\sqrt {2}}{4}&0&\frac{\sqrt {2}}{4}&0\\ \noalign{\medskip}
\frac{\sqrt {2}}{4}&-\frac{\sqrt {2}}{4}&\frac{7\sqrt {2}}{4}&\frac{\sqrt {2}}{4}&-\frac{
\sqrt {2}}{2}&\frac{\sqrt {2}}{2}\\ \noalign{\medskip}-\sqrt {2}&0&\frac{\sqrt {
2}}{4}&\sqrt {2}&0&0\\ \noalign{\medskip}0&\frac{\sqrt {2}}{4}&-\frac{\sqrt {2}}{2}&0
&\frac{\sqrt {2}}{2}&0\\ \noalign{\medskip}\sqrt {2}&0&\frac{\sqrt {2}}{2}&0&0&
\sqrt {2}\end {array} \right]. 
$$
and it is positive definite.  Therefore, we can conclude that $(\omega,\sigma)$ is a coupled $\SU(3)$-structure on $\frak{n}_9$.  

For $\frak{n}_{28}$  consider the pair of compatible, normalized, stable forms
\begin{equation} \label{copledn28}
\left(\omega = e^{12} + e^{34} - e^{56}, \quad  \sigma = e^{136}-e^{145}-e^{235}-e^{246}\right).
\end{equation}
This pair defines a coupled $\SU(3)$-structure with  $d \omega = -\sigma$. 
Moreover, the associated inner product 
$$
h = (e^1)^2 + \ldots + (e^6)^2
$$
is a  nilsoliton with
$${\rm Ric}(h) = -3I + 2 \, {\mbox{diag}} (1,1,1,1,2,2).$$
Summarizing our results, we can conclude that  $\frak n_9$ and $\frak n_{28}$  are, up to isomorphisms,  the only six-dimensional nilpotent Lie 
algebras admitting a coupled $\SU(3)$-structure.   

\begin{table} [ht]
\caption{Expression of $\lambda(\sigma)$ for the  six-dimensional nilpotent Lie algebras admitting a half-flat $\SU(3)$-structure.}
\label{tab:tabella}
\centering
\renewcommand\arraystretch{1.18}
\begin{tabular}{|c|c|c|c|}
\hline
$\mathfrak{n}_\cdot$& $(de^1,de^2,de^3,de^4,de^5,de^6)$ &$\lambda(\sigma)$&${\rm Sign\ of\ }  \lambda(\sigma)$\\ \hline
$\mathfrak{n}_4$       &$(0,0,e^{12},e^{13}, e^{14}+e^{23}, e^{24}+e^{15})$& $4c^4b_{15}^2(-b_{15}(b_{12}+b_{13})+b_{14}^2)$ & $?$   \\ \hline 
$\mathfrak{n}_6$       &$(0,0,e^{12},e^{13},e^{23},e^{14})$&$c^4b_{15}^4$ 						                & $\geq 0$ \\ \hline 
$\mathfrak{n}_7$       &$(0,0,e^{12},e^{13},e^{23},e^{14}-e^{25})$&$c^4(b_{14}^2-b_{15}^2)^2$   				       & $\geq 0$    \\ \hline 
$\mathfrak{n}_8$       &$(0,0,e^{12},e^{13},e^{23},e^{14}+e^{25})$&$c^4(b_{14}^2-b_{15}^2)^2$                                          & $\geq 0$   \\ \hline 
$\mathfrak{n}_9$       &$(0,0,0,e^{12},e^{14}-e^{23},e^{15}+e^{34})$&$4c^4b_{15}^2(-b_{15}(b_{9}+b_{13})+b_{14}^2)$    & $?$  \\ \hline 
$\mathfrak{n}_{10}$  &$(0,0,0,e^{12},e^{14},e^{15}+e^{23})$&$c^4b_{15}^4$  							       & $\geq 0$  \\ \hline 
$\mathfrak{n}_{11}$  &$(0,0,0,e^{12},e^{14},e^{15}+e^{23}+e^{24})$&$c^4b_{15}^4$ 							       & $\geq 0$  \\ \hline 
$\mathfrak{n}_{12}$  &$(0,0,0,e^{12},e^{14},e^{15}+e^{24})$&$0$  								       & $0$  \\ \hline 
$\mathfrak{n}_{13}$  &$(0,0,0,e^{12},e^{14},e^{15})$&$0$  								       & $0$  \\ \hline 
$\mathfrak{n}_{14}$  &$(0,0,0,e^{12},e^{13},e^{14}+e^{35})$&$c^4b_{14}^4$ 							       & $\geq 0$  \\ \hline 
$\mathfrak{n}_{15}$  &$(0,0,0,e^{12},e^{23},e^{14}+e^{35})$&$c^4(b_{14}^2-b_{15}^2)^2$  				       & $\geq 0$  \\ \hline 
$\mathfrak{n}_{16}$  &$(0,0,0,e^{12},e^{23},e^{14}-e^{35})$&$c^4(b_{14}^2+b_{15}^2)^2$  				       & $\geq 0$  \\ \hline 
$\mathfrak{n}_{21}$  &$(0,0,0,e^{12},e^{13},e^{14}+e^{23})$&$0$  								       & $0$  \\ \hline 
$\mathfrak{n}_{22}$  &$(0,0,0,e^{12},e^{13},e^{24})$&$c^4b_{15}^4$  							       & $\geq 0$  \\ \hline 
$\mathfrak{n}_{24}$  &$(0,0,0,e^{12},e^{13},e^{23})$&$0$  								       & $0$  \\ \hline 
$\mathfrak{n}_{25}$  &$(0,0,0,0,e^{12},e^{15}+e^{34})$&$c^4b_{15}^4$  						       & $\geq 0$  \\ \hline 
$\mathfrak{n}_{27}$  &$(0,0,0,0,e^{12},e^{14}+e^{25})$&$0$  								       & $0$  \\ \hline 
$\mathfrak{n}_{28}$  &$(0,0,0,0,e^{13}-e^{24},e^{14}+e^{23})$&$-4c^4b_{15}^4$  						       & $\leq 0$  \\ \hline 
$\mathfrak{n}_{29}$  &$(0,0,0,0,e^{12},e^{14}+e^{23})$&$0$  								       & $0$  \\ \hline 
$\mathfrak{n}_{30}$  &$(0,0,0,0,e^{12},e^{34})$&		$c^4b_{15}^4$  							       & $\geq 0$  \\ \hline 
$\mathfrak{n}_{31}$  &$(0,0,0,0,e^{12},e^{13})$&		$0$  								       & $0$  \\ \hline 
$\mathfrak{n}_{32}$  &$(0,0,0,0,0,e^{12}+e^{34})$&		$0$ 								       & $0$  \\ \hline 
$\mathfrak{n}_{33}$  &$(0,0,0,0,0,e^{12})$&			$0$  								       & $0$  \\ \hline 
$\mathfrak{n}_{34}$  &$(0,0,0,0,0,0)$&				$0$  								       & $0$  \\ \hline 
\end{tabular} 
\end{table}
\renewcommand\arraystretch{1}

We have just provided  a coupled $\SU(3)$-structure on $\frak{n}_{28}$ whose associated inner product  is a nilsoliton, we claim that this is the unique case among all 
six-dimensional nilpotent Lie algebras. 
It is clear that to prove the previous assertion it suffices to show that $\frak{n}_9$ does not admit any coupled $\SU(3)$-structure inducing a nilsoliton inner product. 
In order to do this, we  consider an orthonormal basis $(e_1,\ldots,e_6)$ of $\frak{n}_9$ whose dual basis satisfies the structure equations
$$\left(0,0,0,\frac{\sqrt{5}}{2}e^{12},e^{14}-e^{23},\frac{\sqrt{5}}{2}e^{15}+e^{34}\right)$$
(by the results of \cite{Lauret3} and \cite{Wi}, these are, up to isomorphisms, the structure equations for which the considered inner product  on $\frak{n}_9$ is a nilsoliton). 
As we did before, consider a generic 2-form $\omega$, the 3-form $\sigma = c(d\omega)$, evaluate $\lambda(\sigma)$ and impose that it is negative. 
Then compute $J_\sigma$ and the matrix associated to $h(\cdot,\cdot) = \omega(J_\sigma\cdot,\cdot)$ with respect to the considered basis. 
Since $h$ has to be the restriction to $\frak{n}_9$ of an Einstein inner product defined on $\frak{n}_9 \oplus \RR e_7$ and since the latter is unique up to scaling, 
we have to impose that the symmetric matrix associated to $h$ is a multiple of the identity. Solving the associated equations we find that $\lambda(\sigma)$ has to be zero, 
which is a contradiction. 
\end{proof}

Starting from a six-dimensional nilpotent Lie algebra $\frak{n}$ endowed with a coupled $\SU(3)$-structure, 
it is possible to construct a locally conformal calibrated $G_2$-structure on the rank-one solvable extension  $\frak{s} = \frak{n} \oplus \RR e_7$ under some extra hypothesis. 
Let $\hat{d}$ denote the exterior derivative on $\frak{n}$ and $d$ denote the exterior derivative on $\frak{s}$. 
Observe that given a $k$-form $\theta \in \Lambda^k(\frak{n}^*)$ we have 
$$
d\theta = \hat{d}\theta + \rho\W e^7
$$
for some $\rho \in \Lambda^k(\frak{n}^*)$.

\begin{prop}\label{eqn:propext}
Let $\frak n$ be a six-dimensional, nilpotent Lie algebra endowed with a coupled $\SU(3)$-structure $(\omega,\sigma)$ 
with $\hat{d} \omega = c \sigma$, $c \in \R - \{ 0 \}$. 
Consider on its rank one solvable extension $\frak{s} = \frak{n} \oplus \RR e_7$  the $G_2$-structure  defined by $\f = \omega \W e^7 + \sigma$, 
where the closed 1-form $e^7$ is the dual of $e_7$. Then the $G_2$-structure is locally conformal calibrated with $\tu=\frac13ce^7$ if and only if $d\sigma = -2c\sigma \W e^7$. 
\end{prop}
\proof
Suppose that $d\sigma = -2c\sigma\W e^7$, we can write $d\omega = \hat{d}\omega + \gamma\W e^7$ for some 2-form $\gamma \in \Lambda^2 (\frak{n}^*)$.  
We  obtain  $d \varphi = ce^7\W\f.$ Then, $\f$ is locally conformal calibrated with $\tu = \frac13ce^7$.

Conversely,  suppose that $\f$ is locally conformal calibrated with $\tu=\frac13ce^7$. Then we have $d\f = ce^7 \W \f$. 
Moreover, we know that $d\sigma = \hat{d}\sigma  + \beta\W e^7 = \beta\W e^7$ for some 3-form $\beta \in \Lambda^3 (\frak{n}^*)$, since $\sigma$ is $\hat{d}$-closed. We then have
$$
d\f  =  d\omega\W e^7 + d\sigma = e^7\W (-c\sigma -\beta)
$$
and comparing this with the previous expression of $d\f$ we obtain
$$
\begin{array}{lclcl}
e^7\W (-c\sigma -\beta) & = & c e^7 \W \f & = & e^7 \W (c\sigma)
\end{array}
$$
from which follows $\beta = -2c\sigma$.\endproof

Now we will construct an Einstein locally conformal  calibrated $G_2$-structure on a rank-one extension of the Lie algebra $\frak n_{28}$ 
(Lie algebra of the $3$-dimensional complex Heisenberg group) endowed  with the coupled $\SU(3)$-structure \eqref{copledn28}.

\begin{ex} \label{eqn:ex2} 
Consider $\frak{n}_{28}$ and the metric rank-one solvable extension $\frak s = \frak{n}_{28} \oplus\RR e_7$  with structure equations
$$
\left(\frac12e^{17},\frac12e^{27},\frac12e^{37},\frac12e^{47}, e^{13} - e^{24}+e^{57},  e^{14} + e^{23}+e^{67},0\right).
$$

The associated solvable Lie group $S$ is  not unimodular and so it does not admit any compact quotient \cite{Mi}.
Consider on $\frak{n}_{28}$ the coupled $\SU(3)$-structure $(\omega, \sigma)$ given by \eqref{copledn28}
with the nilsoliton associated inner product
$$
h = (e^1)^2 + \ldots + (e^6)^2.
$$
Then the inner product on $\frak{s}$
$$
g = (e^1)^2 + \ldots + (e^7)^2
$$
is Einstein with  Ricci tensor  ${\mbox Ric} (g) = - 3 g$.

Since $d\sigma = 2\sigma\W e^7$, by the previous proposition we have a locally conformal calibrated $G_2$-structure on $\frak{s}$ given by
$$
\varphi = \omega\W e^7 +\sigma = e^{127} + e^{347} - e^{567} + e^{136} - e^{145} - e^{235} - e^{246}
$$
and it is easy to show that $g_{\varphi} = g$.  
Then the corresponding solvmanifold $(S, \varphi)$ is an example of  non-compact homogeneous manifold  endowed with an Einstein (non-flat) locally conformal calibrated 
$G_2$-structure.

Observe that the $G_2$-structure $\varphi$ satisfies the conditions 
$$
\begin{array} {l}
d \varphi = - e^7 \wedge \varphi,\\
d * \varphi = -  e^7 \wedge (3 e^{1256} + 2 e^{1234} + 3 e^{3456}).
\end{array}
$$
Then 
$$
\tau_1 = - \frac{1} {3} e^7,
$$ 
as we expected from Proposition \ref{eqn:propext}, and 
$$
\tau_2 =  -\left( \frac{5}{3}  e^{12} + \frac{5} {3} e^{34} + \frac{10} {3} e^{56} \right).
$$
Moreover, the $G_2$-structure  is not $*$-Einstein, since by direct computation with respect to the orthonormal  basis $\left(e_1, \ldots, e_7\right)$ one has
$$
\rho^* = \left( \begin{array} {ccccccc}   1&0&0&0&0&0&0\\  0& 1&0&0&0&0&0\\  0&0& 1&0&0&0&0\\ 0&0&0&1&0&0&0\\ 
0&0&0&0&22&0&0\\ 0&0&0&0&0&22&0\\ 0&0&0&0&0&0&-6
\end{array} \right).
$$
\end{ex}
 
It is worth emphasizing here that, by \cite{FFM}, on seven-dimensional solvmanifolds there are no left-invariant calibrated $G_2$-structures inducing 
an Einstein non-flat metric.The previous example shows that the situation is different in the case of locally conformal calibrated $G_2$-structures.
 
We provide now a non-compact  example of  homogeneous manifold admitting an  Einstein (non-flat)  locally conformal parallel $G_2$-structure.

\begin{ex} The Einstein rank-one solvable extension of the six-dimensional abelian Lie algebra
is the  solvable Lie algebra with structure equations
$$(ae^{17},ae^{27},ae^{37},ae^{47},ae^{57},ae^{67},0),$$
where $a$ is a nonzero real number. 
The Riemannian metric
$$
g = (e^1)^2 + \ldots + (e^7)^2
$$
is Einstein  with  Ricci tensor given by ${\mbox Ric} (g) = - 6  a^2 g$.

The 3-form
$$
\varphi = -e^{125} - e^{136}- e^{147} + e^{237}  - e^{246} + e^{345} -e^{567}
$$
has stabilizer $G_2$, is such that $g_\f = g$ and satisfies the conditions 
$$
\begin{array} {l}
d \varphi = -3 a e^{2467} + 3 a e^{3457} - 3 a e^{1257} - 3 a e^{1367},\\
d * \varphi = 4 a e^{23567} + 4 a e^{12347} - 4 a e^{14567}.
\end{array}
$$
It is immediate to show that $\tu = -ae^7 $ and $\tz\equiv0, \td\equiv0, \ttr\equiv0$, that is, the $G_2$-structure $\f$ is locally conformal parallel.
\end{ex}

\end{document}